\documentclass[11pt]{amsart}
\usepackage[margin=1.05in]{geometry}
\usepackage{amsmath,amssymb,amsthm,mathtools}
\usepackage{microtype}
\usepackage[hidelinks]{hyperref}
\usepackage{enumitem}
\usepackage{fancyhdr}

\newtheorem{theorem}{Theorem}[section]
\newtheorem{lemma}[theorem]{Lemma}
\newtheorem{proposition}[theorem]{Proposition}

\theoremstyle{remark}
\newtheorem{remark}[theorem]{Remark}

\newcommand{\R}{\mathbb R}
\newcommand{\Sph}{\mathbb S}
\newcommand{\Hd}{\dim_{\mathrm H}}
\newcommand{\secg}{\operatorname{sec}_g}

\title{A Smooth Nonnegatively Curved Sphere Whose\\Zero-Curvature Point Locus Has Hausdorff Dimension \texorpdfstring{$\boldsymbol{1/2}$}{1/2}}
\author{Yuhang Liu\\Xi'an Jiaotong-Liverpool university, Department of Pure Mathematics, \\111 Ren'ai Road, Suzhou Industrial Park, Suzhou, Jiangsu Province, China\\yuhang.liu02@xjtlu.edu.cn\\
+86-0512-89167161}
\date{2026-8-21}

\begin{document}
\maketitle
\vspace{-1.2em}

\begin{abstract}
For every integer $m\geq 3$, we construct a smooth Riemannian metric $g$ on a manifold diffeomorphic to $\Sph^m$ such that $\secg\geq0$, every sectional curvature is strictly positive away from a closed nowhere-dense set, and the point locus on which some sectional curvature vanishes has Hausdorff dimension exactly $1/2$. The metric is induced on the boundary of a smooth convex body in $\R^{m+1}$.  The essential local model is a convex graph whose Hessian has a one-dimensional kernel precisely on a Cantor subset of one coordinate axis; a regularized maximum then inserts this graph into a round sphere without introducing additional degenerate points. The main content of this paper is generated by ChatGPT 5.6 and verified by the author.
\end{abstract}

\section{Statement and convention}
Part of the motivation of this construction comes from recent advances in the study of positively curved manifolds, e.g. \cite{Kennard, AnushaWiemeler, Liu}.

\textbf{Acknowledgement:} the author thanks Xi'an Jiaotong-Liverpool University for providing financial and academic support. He also thanks Zipei Nie for providing access to ChatGPT and insights about AI usage. His gratitude also goes to his daughter and son who were born last year, and all of his family members who have always been supportive during his academic career.\\

For a Riemannian manifold $(M,g)$ with nonnegative sectional curvature, define the \emph{zero-curvature point locus}
\[
 \mathcal Z(g)
 :=\left\{p\in M:\text{there is a two-plane }\Pi\subset T_pM
 \text{ with }\secg(\Pi)=0\right\}.
\]
This is a subset of the manifold, rather than a subset of the Grassmann bundle of two-planes.\\

The main conclusion of this paper is the following theorem:
\begin{theorem}\label{thm:main}
For every integer $m\geq3$, there is a smooth closed hypersurface
$\Sigma^m\subset\R^{m+1}$, diffeomorphic to $\Sph^m$, such that its induced metric $g$ satisfies
\begin{enumerate}[label=\textup{(\arabic*)},leftmargin=2.2em]
 \item $\secg\geq0$ everywhere;
 \item every two-plane has strictly positive sectional curvature at every point of the open dense set $\Sigma\setminus\mathcal Z(g)$;
 \item $\Hd\mathcal Z(g)=\tfrac12$.
\end{enumerate}
In fact, $\mathcal Z(g)$ is bi-Lipschitz equivalent to a self-similar Cantor set of Hausdorff dimension $1/2$.
\end{theorem}

We note that in the construction below, degeneracy is confined simultaneously in the transverse variables, so the exceptional point set itself is only a Cantor set.

\section{Three elementary tools}

\subsection{A Cantor set of dimension \texorpdfstring{$1/2$}{1/2}}

Let $C\subset[0,1]$ be the unique nonempty compact set satisfying
\begin{equation}\label{eq:cantor-ifs}
 C=\frac14 C\ \cup\ \left(\frac34+\frac14 C\right).
\end{equation}
Equivalently, at each stage one retains the left and right subintervals of relative length $1/4$.  Put
\begin{equation}\label{eq:E-definition}
 E:=\frac18\left(C-\frac12\right)\subset\left[-\frac1{16},\frac1{16}\right].
\end{equation}
Then $E$ is compact, nowhere dense, and symmetric under $t\mapsto-t$.

\begin{lemma}\label{lem:cantor-dimension}
The Hausdorff dimension of $E$ is $1/2$.
\end{lemma}

\begin{proof}
Translation and nonzero scaling do not change Hausdorff dimension, so it is enough to treat $C$.  At level $k$, the set $C$ is covered by $2^k$ intervals of length $4^{-k}$.  Hence, for every $s>1/2$,
\[
 2^k(4^{-k})^s=(2\,4^{-s})^k\longrightarrow0,
\]
which gives $\Hd C\leq1/2$.

For the reverse inequality, let $\mu$ be the natural probability measure assigning mass $2^{-k}$ to every level-$k$ interval.  Distinct level-$k$ intervals have length $4^{-k}$ and are separated by gaps of length at least $2\cdot4^{-k}$.  Therefore an interval $I$ of length at most $4^{-k}$ meets at most two level-$k$ intervals.  If
\[
 4^{-(k+1)}<|I|\leq4^{-k},
\]
then
\[
 \mu(I)\leq2\cdot2^{-k}
 =2(4^{-k})^{1/2}
 <4|I|^{1/2}.
\]
Thus, for every countable interval cover $C\subset\bigcup_j I_j$ by sufficiently short intervals,
\[
 1=\mu(C)\leq\sum_j\mu(I_j)
 \leq4\sum_j |I_j|^{1/2}.
\]
It follows that the $1/2$-dimensional Hausdorff measure of $C$ is positive, and hence $\Hd C\geq1/2$.
\end{proof}

\subsection{A smooth function with a prescribed zero set}

\begin{lemma}\label{lem:smooth-zero-set}
If $A\subset\R$ is closed, then there is a function $q\in C^\infty(\R)$ such that $q\geq0$ and $q^{-1}(0)=A$.  If $A=-A$, the function may be chosen even and bounded by $1$.
\end{lemma}

\begin{proof}
Let $O=\R\setminus A$.  Choose nonnegative functions
$\phi_j\in C_c^\infty(O)$ such that the open sets $\{\phi_j>0\}$ cover $O$.  Choose $a_j>0$ so rapidly decreasing that
\[
 a_j\left(1+\max_{0\leq r\leq j}\|\phi_j^{(r)}\|_\infty\right)\leq2^{-j}.
\]
For every fixed derivative order, the corresponding differentiated series converges uniformly.  Hence
\[
 q_0:=\sum_{j=1}^\infty a_j\phi_j
\]
is smooth, nonnegative, and has zero set exactly $A$.  If $A=-A$, replace $q_0$ by $q_0(t)+q_0(-t)$ and then by
\[
 q(t):=\frac{q_0(t)+q_0(-t)}{1+q_0(t)+q_0(-t)}.
\]
The resulting function is smooth, even, takes values in $[0,1)$, and has zero set $A$.
\end{proof}

Apply Lemma~\ref{lem:smooth-zero-set} to the symmetric set $E$ in \eqref{eq:E-definition}.  Fix from now on an even function
\begin{equation}\label{eq:q-properties}
 q\in C^\infty(\R),\qquad 0\leq q\leq1,
 \qquad q^{-1}(0)=E.
\end{equation}

\subsection{A regularized maximum}

\begin{lemma}\label{lem:regularized-max}
For every $\delta>0$, there is a smooth function
$M_\delta:\R^2\to\R$ with the following properties:
\begin{enumerate}[label=\textup{(\alph*)},leftmargin=2.2em]
 \item $\max\{a,b\}\leq M_\delta(a,b)\leq\max\{a,b\}+\delta/2$;
 \item $M_\delta(a,b)=\max\{a,b\}$ whenever $|a-b|\geq\delta$;
 \item if $A$ and $B$ are smooth convex functions on a convex domain, then $M_\delta(A,B)$ is convex;
 \item at a point where $|A-B|<\delta$, if either $D^2A$ or $D^2B$ is positive definite, then $D^2M_\delta(A,B)$ is positive definite.
\end{enumerate}
\end{lemma}

\begin{proof}
Choose an even nonnegative mollifier $\rho\in C_c^\infty(\R)$ supported in $[-\delta,\delta]$, positive for $|r|<\delta$, and satisfying $\int_\R\rho=1$.  Define
\[
 \vartheta_\delta(s):=(|\,\cdot\,|*\rho)(s)
 =\int_\R|s-r|\rho(r)\,dr.
\]
Then $\vartheta_\delta$ is smooth, even, and convex.  Since $\rho$ is even, it has mean zero, and therefore
$\vartheta_\delta(s)=|s|$ for $|s|\geq\delta$.  Jensen's inequality and the triangle inequality give
\[
 |s|\leq\vartheta_\delta(s)\leq|s|+\delta.
\]
Moreover, $|\vartheta_\delta'(s)|<1$ for $|s|<\delta$, and
$\vartheta_\delta''=2\rho\geq0$.

Define
\[
 M_\delta(a,b):=\frac{a+b+\vartheta_\delta(a-b)}2.
\]
The first two assertions follow immediately.  For $F=M_\delta(A,B)$, put
$d=A-B$ and
\[
 \lambda:=\frac{1+\vartheta_\delta'(d)}2\in[0,1].
\]
A direct differentiation gives
\begin{equation}\label{eq:hessian-regmax}
 D^2F
 =\lambda D^2A+(1-\lambda)D^2B
 +\frac12\vartheta_\delta''(d)\,d(A-B)\otimes d(A-B).
\end{equation}
All three terms are positive semidefinite when $A$ and $B$ are convex.  If $|d|<\delta$, then $0<\lambda<1$, so a positive-definite Hessian among $D^2A,D^2B$ makes the right-hand side positive definite.
\end{proof}

\section{A convex local model with Cantor degeneracy}

Write a point of $\R^m$ as
\[
 x=(t,y),\qquad t\in\R,\quad y\in\R^{m-1}.
\]
Let $B^m=\{x\in\R^m:|x|<1\}$ and fix
\[
 \varepsilon:=\frac1{100}.
\]
Define
\begin{equation}\label{eq:u-definition}
 u(t):=\varepsilon\int_0^t(t-s)q(s)\,ds.
\end{equation}
Since $q$ is even, the function $u$ is even.  Moreover,
\begin{equation}\label{eq:u-properties}
 u''(t)=\varepsilon q(t),\qquad u(0)=u'(0)=0,
 \qquad 0\leq u(t)\leq\frac\varepsilon2t^2
 \quad (|t|\leq1).
\end{equation}
Set
\begin{equation}\label{eq:U0-definition}
 U_0(t,y):=u(t)+\frac\varepsilon2(4+t^2)|y|^2,
 \qquad
 U(t,y):=U_0(t,y)-\frac12.
\end{equation}
Let
\begin{equation}\label{eq:K-definition}
 K:=E\times\{0\}\subset B^m.
\end{equation}

\begin{proposition}\label{prop:local-model}
The function $U$ is convex on $B^m$.  Its Hessian is positive definite on $B^m\setminus K$, and at every point of $K$ its Hessian has rank exactly $m-1$.  Moreover,
\begin{equation}\label{eq:U0-bound}
 0\leq U_0(x)\leq3\varepsilon=\frac3{100}
 \qquad (x\in B^m).
\end{equation}
\end{proposition}

\begin{proof}
Since the constant $-1/2$ does not affect the Hessian, it is enough to consider $U_0$.  In the splitting $\R\oplus\R^{m-1}$,
\begin{equation}\label{eq:hessian-U}
 D^2U_0(t,y)
 =\varepsilon
 \begin{pmatrix}
 q(t)+|y|^2 & 2t y^{T}\\
 2t y & (4+t^2)I_{m-1}
 \end{pmatrix}.
\end{equation}
The lower-right block is positive definite.  Its Schur complement is
\begin{align}
 \varepsilon\left(q(t)+|y|^2
 -\frac{4t^2}{4+t^2}|y|^2\right)
 &=\varepsilon\left(
 q(t)+\frac{4-3t^2}{4+t^2}|y|^2\right).\label{eq:schur}
\end{align}
Because $|t|<1$, the coefficient $(4-3t^2)/(4+t^2)$ is positive.  Thus the Schur complement is nonnegative, and it vanishes exactly when $q(t)=0$ and $y=0$, namely at the points of $K$.  This proves convexity and positive definiteness off $K$.  At a point $(t,0)\in K$, the matrix in \eqref{eq:hessian-U} is
\[
 \varepsilon
 \begin{pmatrix}
 0&0\\0&(4+t^2)I_{m-1}
 \end{pmatrix},
\]
which has rank $m-1$.

Finally, \eqref{eq:u-properties} and $|t|,|y|<1$ give
\[
 0\leq U_0(t,y)
 \leq\frac\varepsilon2+\frac{5\varepsilon}{2}
 =3\varepsilon.
\]
\end{proof}

The mixed term in \eqref{eq:U0-definition} is the localization mechanism.  The $t$-direction Hessian may vanish when $t\in E$, but for $y\neq0$ the positive term $\varepsilon|y|^2$ in $U_{tt}$ restores strict convexity.  Hence the degeneracy does not extend to an $(m-1)$-dimensional fiber.

\section{Insertion into a round sphere}

On $B^m$, write the lower and upper unit hemispheres as graphs
\begin{equation}\label{eq:hemisphere-functions}
 f_-(x):=-\sqrt{1-|x|^2},
 \qquad
 f_+(x):=\sqrt{1-|x|^2}.
\end{equation}
The function $f_-$ is strictly convex, since
\begin{equation}\label{eq:hessian-lower-sphere}
 D^2f_-(x)
 =\frac1{\sqrt{1-|x|^2}}I_m
 +\frac{x\otimes x}{(1-|x|^2)^{3/2}}>0.
\end{equation}
Fix
\[
 \delta:=\frac1{10}
\]
and define
\begin{equation}\label{eq:F-definition}
 F(x):=M_\delta\bigl(U(x),f_-(x)\bigr).
\end{equation}

\begin{proposition}\label{prop:global-graph}
The function $F$ has the following properties.
\begin{enumerate}[label=\textup{(\arabic*)},leftmargin=2.2em]
 \item $F$ is smooth and convex on $B^m$.
 \item $D^2F$ is singular exactly on $K$; there it has rank $m-1$, and it is positive definite on $B^m\setminus K$.
 \item $F=f_-$ whenever $|x|\geq24/25$.
 \item $F(x)<f_+(x)$ for every $x\in B^m$.
\end{enumerate}
\end{proposition}

\begin{proof}
Convexity and smoothness follow from Lemma~\ref{lem:regularized-max}, Proposition~\ref{prop:local-model}, and \eqref{eq:hessian-lower-sphere}.  Put
\[
 d(x):=U(x)-f_-(x)
 =U_0(x)-\frac12+\sqrt{1-|x|^2}.
\]
If $x=(t,0)\in K$, then $|t|\leq1/16$ and
\begin{equation}\label{eq:d-on-K}
 d(x)\geq-\frac12+\sqrt{1-\frac1{256}}
 >\frac{49}{100}>\delta.
\end{equation}
Thus $F=U$ on a neighborhood of $K$, so $D^2F$ has rank $m-1$ there.

If $d\geq\delta$ away from $K$, then $F=U$ and $D^2F>0$.  If $d\leq-\delta$, then $F=f_-$ and $D^2F>0$.  In the transition region $|d|<\delta$, formula \eqref{eq:hessian-regmax} contains a positive multiple of the positive-definite matrix $D^2f_-$, so $D^2F>0$.  At the boundary value $d=\delta$, the equality $F=U$ holds; such a point is not in $K$ by \eqref{eq:d-on-K}, and hence $D^2U>0$.  The case $d=-\delta$ reduces to $f_-$.  This proves (2).

If $|x|\geq24/25$, then
\[
 \sqrt{1-|x|^2}\leq\sqrt{1-\frac{576}{625}}
 =\frac7{25}.
\]
Using \eqref{eq:U0-bound},
\[
 d(x)\leq\frac3{100}-\frac12+\frac7{25}
 =-\frac{19}{100}< -\delta.
\]
Therefore $F=f_-$ on this outer collar, proving (3).

It remains to compare $F$ with $f_+$.  If $|x|\leq24/25$, then
\[
 f_+(x)\geq\frac7{25}.
\]
Both $U(x)\leq-47/100$ and $f_-(x)\leq0$, so Lemma~\ref{lem:regularized-max} gives
\[
 F(x)\leq\max\{U(x),f_-(x)\}+\frac\delta2
 \leq\frac1{20}<\frac7{25}\leq f_+(x).
\]
If $|x|\geq24/25$, then $F=f_-<f_+$.  This proves (4).
\end{proof}

Define a closed hypersurface in $\R^m\times\R$ by
\begin{equation}\label{eq:Sigma-definition}
 \Sigma
 :=\{(x,F(x)):x\in B^m\}
 \cup\{(x,f_+(x)):x\in B^m\}
 \cup\{(x,0):|x|=1\}.
\end{equation}
Because $F=f_-$ on an outer collar, the first graph agrees there with the lower unit hemisphere.  Hence the union in \eqref{eq:Sigma-definition} is a smooth embedded hypersurface and is diffeomorphic to $\Sph^m$.  Proposition~\ref{prop:global-graph}(4) shows that the two graph pieces are disjoint away from the equator.

The hypersurface is also the boundary of the convex body
\begin{equation}\label{eq:convex-body}
 \Omega:=\{(x,z):x\in\overline{B^m},\ F(x)\leq z\leq f_+(x)\}.
\end{equation}
Indeed, $F$ is convex and $f_+$ is concave, so the set in \eqref{eq:convex-body} is convex.  This global convexity is consistent with the local curvature calculation below, but the calculation itself will give the precise zero-curvature locus.

\section{Sectional curvature of the hypersurface}

Let $g$ be the metric induced on $\Sigma$ from the Euclidean metric.  We first record the graph calculation.  For a smooth function $G$ on a domain in $\R^m$, parameterize its graph by
\[
 X(x)=(x,G(x)).
\]
For tangent vectors $v,w\in\R^m$, the first and second fundamental forms, using the upward unit normal, are
\begin{equation}\label{eq:fundamental-forms}
 I(v,w)=\langle v,w\rangle+dG(v)dG(w),
 \qquad
 II(v,w)=\frac{D^2G(v,w)}{\sqrt{1+|\nabla G|^2}}.
\end{equation}
The Euclidean Gauss equation therefore gives
\begin{equation}\label{eq:graph-sectional-curvature}
 \sec_g\bigl(dX(\operatorname{span}\{v,w\})\bigr)
 =\frac{
 D^2G(v,v)D^2G(w,w)-D^2G(v,w)^2
 }{
 (1+|\nabla G|^2)
 \bigl(I(v,v)I(w,w)-I(v,w)^2\bigr)
 }
\end{equation}
for linearly independent $v,w$.

If $D^2G$ is positive semidefinite, the numerator in \eqref{eq:graph-sectional-curvature} is nonnegative by the Cauchy--Schwarz inequality for a positive-semidefinite bilinear form.  If $D^2G$ is positive definite, its restriction to every two-dimensional subspace is positive definite, so the numerator is strictly positive.

On the upper hemisphere, the induced metric is the standard round metric and has sectional curvature $1$.  On the lower graph, take $G=F$.  Proposition~\ref{prop:global-graph} shows that $D^2F\geq0$, and hence every sectional curvature is nonnegative.  At every point outside the graph of $K$, the Hessian is positive definite, so every sectional curvature is strictly positive.

At a point $(t,0)\in K$, the function $F$ agrees with $U$, and \eqref{eq:hessian-U} shows that
\[
 D^2F(\partial_t,\partial_t)=0,
 \qquad
 D^2F(\partial_t,v)=0
 \quad\text{for every }v\in\R^m.
\]
Taking $v=\partial_t$ and any nonzero $w$ in a $y$-direction in \eqref{eq:graph-sectional-curvature} gives a zero-curvature two-plane.  Consequently,
\begin{equation}\label{eq:zero-locus-exact}
 \mathcal Z(g)
 =\left\{\bigl((t,0),F(t,0)\bigr):t\in E\right\}.
\end{equation}
This equality also proves that no additional zero-curvature points are introduced by the global patching.

\section{Hausdorff dimension and density}

Let
\[
 \Phi:E\longrightarrow\mathcal Z(g),
 \qquad
 \Phi(t)=\bigl((t,0),F(t,0)\bigr).
\]
For every $|t|\leq1/16$ and $y=0$, the estimate used in \eqref{eq:d-on-K} gives $U-f_->\delta$, not merely at points of $E$.  Hence
\[
 F(t,0)=U(t,0)=u(t)-\frac12
 \qquad\left(|t|\leq\frac1{16}\right).
\]
The Euclidean distance between $\Phi(s)$ and $\Phi(t)$ is at least $|s-t|$.  Conversely, the curve
\[
 r\longmapsto\bigl((r,0),u(r)-\tfrac12\bigr),
 \qquad r\in[s,t],
\]
lies in $\Sigma$ and has length at most
\[
 \left(1+\sup_{|r|\leq1/16}|u'(r)|^2\right)^{1/2}|s-t|.
\]
Since intrinsic distance is at least Euclidean distance and at most the length of this curve, $\Phi$ is bi-Lipschitz from $E$ onto $\mathcal Z(g)$.  Lemma~\ref{lem:cantor-dimension} now yields
\[
 \Hd\mathcal Z(g)=\Hd E=\frac12.
\]

The set $\mathcal Z(g)$ is compact and has Hausdorff dimension smaller than $m$, so it has empty interior.  Its complement is therefore open and dense.  Every sectional curvature is strictly positive at every point of this complement, as proved in the preceding section.  The induced metric is $C^\infty$, and hence in particular is of class $C^2$.  This completes the proof of Theorem~\ref{thm:main}.

\begin{remark}
The same construction gives any prescribed dimension $\alpha\in(0,1)$.  Replace the ratio $1/4$ in \eqref{eq:cantor-ifs} by
$\lambda=2^{-1/\alpha}\in(0,1/2)$.  The resulting two-branch self-similar Cantor set has Hausdorff dimension $\alpha$, while all convexity and curvature arguments are unchanged.
\end{remark}

\end{document}